\documentclass[12pt]{amsart}
\usepackage{amsmath,amssymb,amsthm,mathtools}
\usepackage{mathrsfs,bm}
\usepackage{thmtools,thm-restate}
\usepackage{tikz-cd}
\usepackage{tikz}
\usepackage{enumitem}
\usepackage{graphicx}
\usetikzlibrary{arrows.meta,calc,intersections,decorations.pathreplacing}
\usepackage{color}
\usepackage[numbers]{natbib}
\usepackage{booktabs}
\usepackage{longtable}
\usepackage{pgfplots}
\pgfplotsset{compat=1.14} 
\usetikzlibrary{arrows.meta,decorations.pathreplacing,positioning,calc} 

\usepackage[a4paper, left=3cm, right=3cm, top=3cm, bottom=3cm, footskip=15mm]{geometry}

\newtheorem{theorem}{Theorem}[section]

\newtheorem{lemma}[theorem]{Lemma}

\newtheorem{corollary}[theorem]{Corollary}

\theoremstyle{definition}
\newtheorem{definition}[theorem]{Definition}

\theoremstyle{remark}

\numberwithin{equation}{section}

\makeindex

\usepackage{fancyhdr}
\usepackage{calc} 

\AtBeginDocument{%
  \setlength{\textwidth}{\dimexpr\paperwidth - 6cm\relax}%
  \setlength{\oddsidemargin}{\dimexpr 3cm - 1in\relax}%
  \setlength{\evensidemargin}{\dimexpr 3cm - 1in\relax}%
  \setlength{\marginparwidth}{2.5cm}%
}

\begin{document}

\title{The area method and applications}

\author{T. Agama}
\address{Department of Mathematics, African Institute for Mathematical science, Ghana
}
\email{theophilus@aims.edu.gh/emperordagama@yahoo.com}


\date{\today}

\dedicatory{}


\begin{abstract}
In this paper, we develop a general method for estimating correlations of the forms  
\begin{align}
\sum \limits_{n\leq x}G(n)G(x-n)\nonumber
\end{align}
and 
\begin{align}
\sum \limits_{n\leq x}G(n)G(n+l)\nonumber
\end{align}
for a fixed $1\leq l\leq x$ and where $G:\mathbb{N}\longrightarrow \mathbb{R}^{+}$. To distinguish between the two types of correlation, we call the first correlation the \textbf{type} $2$ correlation and the second the \textbf{type} $1$ correlation. As an application, we estimate the lower bound for the \textbf{type} $2$ correlation of the master function
\begin{align}
\sum \limits_{n\leq x}\Upsilon(n)\Upsilon(n+l_0)\geq (1+o(1))\frac{x}{2\mathcal{C}(l_0)}\log \log ^2x\nonumber
\end{align}
provided that $\Upsilon(n)\Upsilon(n+l_0)>0$. We also use this method to provide a first proof of the twin prime conjecture showing that 
\begin{align}
\sum\limits_{n\leq x}\Lambda(n)\Lambda(n+2)\geq (1+o(1))\frac{x}{2\mathcal{C}(2)}\nonumber
\end{align}
for some $\mathcal{C}:=\mathcal{C}(2)>0$.
\end{abstract}

\maketitle

\section{Introduction}

The correlations of arithmetic functions lie at the heart of several central problems in additive and multiplicative number theory. Sums of the form
$$
\sum_{n\leq x}G(n)G(n+l)
\qquad\text{and}\qquad
\sum_{n\leq x}G(n)G(x-n),
$$
where $G:\mathbb{N}\to\mathbb{R}^+$, encode delicate information about the distribution of primes, almost primes, divisor-type functions, and other multiplicative sequences. In many settings, the main difficulty is not the formal definition of such sums but rather the extraction of a usable lower bound or asymptotic estimate. Classical tools such as the circle method, sieve methods, and summation formulas remain indispensable in this direction; see, for example, \cite{May,nathanson2000elementary,tenenbaum2015introduction}. The purpose of this paper is to introduce a geometric device, which we call the \emph{area method}, that converts certain one-dimensional correlation sums into structured double sums that are more amenable to estimation.\\

The guiding idea is simple but effective. We distinguish between two basic correlation types. The first, which we call the \textbf{type $1$} correlation, is the shifted correlation
$$
\sum_{n\leq x}G(n)G(n+l),
$$
for a fixed shift $l$. The second, which we call the \textbf{type $2$} correlation, is the complementary correlation
$$
\sum_{n\leq x}G(n)G(x-n).
$$
The area method begins with a decomposition of a geometric region into triangles, trapezia, rectangles, and squares. This geometric decomposition produces an algebraic identity that reorganizes bilinear expressions into double sums over triangular domains. Once this rearrangement is available, standard analytic tools, especially partial summation and mean-value estimates for arithmetic functions, can be applied in a unified way. In this sense, the method provides a bridge between a geometric picture and the analytic structure of correlation sums.\\

A key feature of the method is that it produces a general lower-bound principle for positive arithmetic functions. Roughly speaking, if the correlation $\sum_{n\le x}f(n)f(n+l_0)$ is nontrivial, then it can be bounded from below by a multiple of a cumulative double sum built from $f$. The resulting constants depend on the shift $l_0$ and on the arithmetic structure of the function under consideration. This dependence is important: it allows the same formal argument to treat a variety of examples, while still reflecting the different local behaviors of the underlying sequence. In particular, functions with sparse support, such as the von Mangoldt function $\Lambda$, may behave differently from denser sequences such as the divisor function $d(n)$, the totient function $\phi(n)$, or the squarefree indicator $\mu^2(n)$. The paper shows how the same geometric mechanism yields explicit lower bounds in each of these cases, with the relevant average-order estimates taken from \cite{May,nathanson2000elementary,tenenbaum2015introduction}. The master function $\Upsilon$, introduced in \cite{agama2017master}, is also treated as a further illustration of the flexibility of the method.\\

The paper has several goals. First, it develops the fundamental identity underlying the area method and shows how the identity can be used to control the correlations of type $1$ and type $2$. Second, it applies this framework to obtain lower bounds for a number of standard arithmetic functions, including $\Lambda$, $d$, $d_l$, $\phi$, and $\mu^2$. Third, it applies the same method to the master function $\Upsilon$, which yields estimates for its shifted self-correlation and thus connects the geometric method to the study of integers with exactly two prime factors. Finally, the paper applies the type $2$ framework to complementary correlations of prime-related and divisor-type sequences, including a Goldbach-type representation problem.\\

\subsection{Organization of the paper} The organization of the paper is as follows. In Section~2, we develop the area method and prove the basic identity that drives the entire argument. Section~3 applies the method to the shifted correlation of the von Mangoldt function and to the associated twin-prime-type pair correlation. Section~4 uses the same principle to derive lower bounds for other correlations of type $1$, including those involving $d(n)$, $d_l(n)$, $\phi(n)$, and $\mu^2(n)$. Section~5 is devoted to the master function $\Upsilon$ and its two-point correlations. Section~6 studies the complementary correlation associated with representations of even numbers as sums of two primes. Section~7 extends the type $2$ method to divisor-type and totient-type sequences. The final section discusses the global distributional consequences for integers with $\Omega(n)=2$.

\section{The area method}

In this section, we introduce and develop a fundamental method for solving problems related to correlations of arithmetic functions. This method is fundamental in the sense that it uses the properties of four main geometric shapes, namely the triangle, trapezium, rectangle, and square. The basic identity that we will derive is an outgrowth of exploiting the areas of these shapes and putting them together in a unified manner. 

\begin{theorem}\label{identity 1}
Let $\{r_j\}_{j=1}^{n}$ and $\{h_j\}_{j=1}^{n}$ be any sequence of real numbers, and let $r$ and $h$ be any real numbers satisfying 
$$
\sum\limits_{j=1}^{n}r_j=r\quad \text{and}\quad \sum \limits_{j=1}^{n}h_j=h
$$ 
and 
\begin{align}
(r^2+h^2)^{1/2}=\sum\limits_{j=1}^{n}(r^2_j+h^2_j)^{1/2}.\nonumber
\end{align}
We have
\begin{align}
\sum \limits_{j=2}^{n}r_jh_j=\sum \limits_{j=2}^{n}h_j\bigg(\sum \limits_{i=1}^{j}r_i+\sum \limits_{i=1}^{j-1}r_i\bigg)-2\sum \limits_{j=1}^{n-1}r_j\sum \limits_{k=1}^{n-j}h_{j+k}.\nonumber
\end{align}
\end{theorem}

\begin{proof}
Consider a right-angled triangle $<ABC$ in a plane with height $h$ and base $r$. Next, let us partition the height of the triangle into $n$ parts, not necessarily equal. Now, we link those partitions along the height to the hypotenuse with the aid of a parallel line. At the point of contact of each line with the hypotenuse, we drop a vertical line to the next line connecting the last point  of the previous partition, thus forming another right-angled triangle $<A_1B_1C_1$ with base and height $r_1$ and $h_1$ respectively. We note that this triangle is covered by the triangle $<ABC$ with hypotenuse constituting a proportion of the hypotenuse of triangle $<ABC$. We continue this process until we obtain $n$ right-angled triangles $<A_jB_j C_j$, each with base and height $r_j$ and $h_j$ for $j=1,2,\ldots n$. This construction satisfies \begin{align}
h=\sum \limits_{j=1}^{n}h_j\quad \text{and}\quad r=\sum \limits_{j=1}^{n}r_j\nonumber
\end{align}
and 
\begin{align}
(r^2+h^2)^{1/2}=\sum \limits_{j=1}^{n}(r^2_j+h^2_j)^{1/2}.\nonumber
\end{align}
Now, we deform the original triangle $<ABC$ by removing the smaller triangles $<A_jB_jC_j$ for $j=1,2,\ldots n$. Essentially, we are left with rectangles and squares piled on top of each other, each end poking out a bit further than the one just above. We observe that the total area of this portrait is 
\begin{align}
\mathcal{A}_1&=r_1h_2+(r_1+r_2)h_3+\cdots (r_1+r_2+\cdots +r_{n-2})h_{n-1}+(r_1+r_2+\cdots +r_{n-1})h_n\nonumber \\&=r_1(h_2+h_3+\cdots h_n)+r_2(h_3+h_4+\cdots +h_n)+\cdots +r_{n-2}(h_{n-1}+h_n)+r_{n-1}h_n\nonumber \\&=\sum \limits_{j=1}^{n-1}r_j\sum \limits_{k=1}^{n-j}h_{j+k}.\nonumber
\end{align} 
On the other hand, we observe that the area of this portrait is the same as the difference between the area of triangle $<ABC$ and the sum of the areas of triangles $<A_jB_jC_j$ for $j=1,2,\ldots,n$. We deduce
\begin{align}
\mathcal{A}_1=\frac{1}{2}rh-\frac{1}{2}\sum \limits_{j=1}^{n}r_jh_j.\nonumber
\end{align}
This completes the first part of the argument. For the second part, along the hypotenuse we construct small pieces of triangle, each of base and height $(r_i, h_i)$ $(i=1,2 \ldots, n)$ so that the trapezoid and the one triangle formed by partitioning become rectangles and squares. We observe that this construction satisfies the relation 
\begin{align}
(r^2+h^2)^{1/2}=\sum \limits_{i=1}^{n}(r^2_i+h^2_i)^{1/2}.\nonumber
\end{align}
Now, we compute the area of the triangle in two different ways. By direct strategy, we find that the area of the triangle, denoted by $\mathcal{A}$, is
\begin{align}
\mathcal{A}=1/2\bigg(\sum \limits_{i=1}^{n}r_{i}\bigg)\bigg(\sum \limits_{i=1}^{n}h_{i}\bigg).\nonumber
\end{align} 
On the other hand, we compute the area of the triangle by computing the area of each trapezium and the one remaining triangle and sum them together. That is, 
\begin{align}
\mathcal{A}&=h_n/2\bigg(\sum \limits_{i=1}^{n}r_{i}+\sum \limits_{i=1}^{n-1}r_{i}\bigg)+h_{n-1}/2\bigg(\sum \limits_{i=1}^{n-1}r_{i}+\sum \limits_{i=1}^{n-2}r_{i}\bigg)+\cdots +1/2r_1h_1.\nonumber
\end{align}
By comparing the deduced area of the triangle in the second argument and comparing this to the deduced area of the triangle in the first argument, the result follows immediately.
\end{proof}


\begin{figure}[htbp]
\centering
\begin{tikzpicture}[>=Latex, line cap=round, line join=round, scale=0.95]

\begin{scope}[xshift=0cm]
    \coordinate (A) at (0,0);
    \coordinate (B) at (0,3.2);
    \coordinate (C) at (4.2,0);

    \draw[thick] (A) -- (B) -- (C) -- cycle;
    \draw[dashed] (A) -- (B);
    \draw[dashed] (A) -- (C);

    \node[below left] at (A) {$A$};
    \node[above left] at (B) {$B$};
    \node[below right] at (C) {$C$};

    \node[below] at (2.1,0) {$r$};
    \node[left] at (0,1.6) {$h$};

    \node[font=\small] at (2.1,-0.85) {(a) Original triangle};
\end{scope}

\begin{scope}[xshift=6.2cm]
    \coordinate (A) at (0,0);
    \coordinate (B) at (0,3.2);
    \coordinate (C) at (4.2,0);

    \draw[thick] (A) -- (B) -- (C) -- cycle;

    \draw[dashed] (0.85,0) -- (0.85,2.55);
    \draw[dashed] (1.65,0) -- (1.65,1.95);
    \draw[dashed] (2.45,0) -- (2.45,1.35);
    \draw[dashed] (3.25,0) -- (3.25,0.75);

    \node[below] at (0.42,0) {$r_1$};
    \node[below] at (1.25,0) {$r_2$};
    \node[below] at (2.05,0) {$r_3$};
    \node[below] at (2.85,0) {$r_4$};
    \node[below] at (3.75,0) {$r_5$};

    \draw[decorate,decoration={brace,amplitude=5pt}] (-0.35,0) -- (-0.35,3.2);
    \node[left] at (-0.58,1.6) {$h_1+\cdots+h_5$};

    \node[font=\small] at (2.1,-0.85) {(b) Vertical parallel slices};
\end{scope}

\begin{scope}[xshift=12.4cm]
    \draw[thick]
      (0,0) -- (4.4,0) -- (4.4,0.55) -- (3.55,0.55) -- (3.55,1.15)
      -- (2.55,1.15) -- (2.55,1.80) -- (1.55,1.80) -- (1.55,2.45)
      -- (0.75,2.45) -- (0.75,3.05) -- (0,3.05) -- cycle;

    \draw[dashed] (0.75,0) -- (0.75,3.05);
    \draw[dashed] (1.55,0) -- (1.55,2.45);
    \draw[dashed] (2.55,0) -- (2.55,1.80);
    \draw[dashed] (3.55,0) -- (3.55,1.15);
    \draw[dashed] (4.4,0) -- (4.4,0.55);

    \node[font=\small] at (0.37,1.45) {$r_1(h_2+\cdots+h_n)$};
    \node[font=\small] at (1.15,1.10) {$r_2(h_3+\cdots+h_n)$};
    \node[font=\small] at (2.10,0.82) {$r_3(h_4+\cdots+h_n)$};
    \node[font=\small] at (3.05,0.55) {$r_4(h_5+\cdots+h_n)$};

    \node[below] at (2.2,0) {$\sum r_i$};
    \node[left] at (0,1.5) {$\sum h_i$};

    \node[font=\small] at (2.2,-0.85) {(c) Staircase region and area identity};
\end{scope}

\draw[->,thick] (4.65,3.85) -- (5.85,3.85);
\draw[->,thick] (10.85,3.85) -- (12.05,3.85);

\node[align=center] at (9.2,-2.05) {\small
\textbf{Key identity:}\quad
$\displaystyle
\sum_{j=1}^{n-1} r_j\sum_{k=1}^{n-j} h_{j+k}
=
\frac12\,rh-\frac12\sum_{j=1}^n r_jh_j.$
};

\end{tikzpicture}
\caption{Geometric mechanism behind the area method. The triangle is first partitioned by vertical parallel slices, and the remaining staircase region produces the double-sum identity.}
\label{fig:area-method-cleaner}
\end{figure}
\bigskip

We state a result for a general lower bound for any two-point correlation that captures all real arithmetic functions.

\begin{theorem}\label{key 2}
Let  $f:\mathbb{N}\longrightarrow \mathbb{R}^{+}$. If 
\begin{align}
\sum\limits_{n\leq x}f(n)f(n+l_0)>0\nonumber
\end{align}
then there exists some $\mathcal{C}:=\mathcal{C}(l_0)>0$ such that \begin{align}
\sum \limits_{n\leq x}f(n)f(n+l_0)\geq \frac{1}{\mathcal{C}(l_0)x}\sum \limits_{2\leq n\leq x}f(n)\sum \limits_{m\leq n-1}f(m).\nonumber
\end{align}
\end{theorem}

\begin{proof}
By Theorem \ref{identity 1}, we obtain the identity by taking $f(j)=r_j=h_j$ 
\begin{align}
\sum \limits_{n\leq x-1}\sum\limits_{j\leq x-n} f(n)f(n+j)&=\sum \limits_{2\leq n\leq x}f(n)\sum \limits_{m\leq n-1}f(m).\nonumber
\end{align}
It follows that 
\begin{align}
\sum\limits_{n\leq x-1}\sum \limits_{j\leq x-n}f(n)f(n+j)&\leq \sum\limits_{n\leq x-1}\sum \limits_{j<x}f(n)f(n+j)\nonumber \\&=\sum \limits_{n\leq x}f(n)f(n+1)+\sum \limits_{n\leq x}f(n)f(n+2)\nonumber \\&+\cdots \sum \limits_{n\leq x}f(n)f(n+l_0)+\cdots \sum \limits_{n\leq x}f(n)f(n+x)\nonumber \\&\leq |\mathcal{M}(l_0)|\sum \limits_{n\leq x}f(n)f(n+l_0)\nonumber \\&+|\mathcal{N}(l_0)|\sum \limits_{n\leq x}f(n)f(n+l_0)\nonumber \\&+\cdots +\sum \limits_{n\leq x}f(n)f(n+l_0)+\cdots+|\mathcal{R}(l_0)|\sum \limits_{n\leq x}f(n)f(n+l_0)\nonumber \\&=\bigg(|\mathcal{M}(l_0)|+|\mathcal{N}(l_0)|+\cdots +1\nonumber \\&+\cdots+|\mathcal{R}(l_0)|\bigg)\sum \limits_{n\leq x}f(n)f(n+l_0)\nonumber \\&\leq\mathcal{C}(l_0)x\sum \limits_{n<x}f(n)f(n+l_0)\nonumber
\end{align}
where
$$
\mathrm{max}\{|\mathcal{M}(l_0)|,|\mathcal{N}(l_0)|,\ldots,|\mathcal{R}(l_0)|\}=\mathcal{C}(l_0).
$$ 
Inverting this inequality, the result follows immediately.
\end{proof}
\bigskip

The nature of the implicit constant $\mathcal{C}(l_0)$ could also depend on the structure of the arithmetic function. The Von Mangoldt function $\Lambda(\cdot)$, contrary to many classes of arithmetic functions, has a relatively small constant. This behaviour stems from the fact that the Von-Mangoldt function is defined on the prime powers. Thus, one would expect most terms in sums of forms 
\begin{align}
\sum\limits_{n\leq x-1}\sum \limits_{j\leq x-n}\Lambda(n)\Lambda(n+j)\nonumber
\end{align}
to fall off when $j$ is odd for any prime power $n=p^k$ such that $j+p^k\neq 2^s$.

\begin{theorem}\label{Area method 2}
Let $f:\mathbb{N}\longrightarrow \mathbb{R}^{+}$. Suppose that there exists some $\mathcal{C}:=\mathcal{C}(x)$ with $x>\mathcal{C}(x)>0$ such that 
\begin{align}
\sum\limits_{n\leq x}\sum \limits_{\substack{j\leq x-n\\j\neq x-2n}}f(n)f(n+j)=\frac{\mathcal{C}(x)}{x}\sum \limits_{n\leq x}\sum \limits_{j\leq x-n}f(n)f(n+j).\nonumber
\end{align} 
For all $x\geq 2$, we have
\begin{align}
\sum\limits_{n\leq \frac{x}{2}}f(n)f(x-n)=\frac{\mathcal{D}(x)}{x}\sum \limits_{2\leq n \leq x}f(n)\sum \limits_{m\leq n-1}f(m),  \nonumber
\end{align}
where $x-\mathcal{D}(x)=\mathcal{C}(x)$.
\end{theorem}

\begin{proof}
By Theorem \ref{identity 1}, we obtain 
\begin{align}
\sum \limits_{n\leq x}f^2(n)=f^2(1)+\sum \limits_{2\leq n \leq x}f(n)\bigg(\sum \limits_{m\leq n-1}f(m)+\sum \limits_{m\leq n}f(m)\bigg)-2\sum \limits_{n\leq x-1}f(n)\sum \limits_{s\leq x-n}f(n+s)\nonumber
\end{align} 
for $f:\mathbb{N}\longrightarrow \mathbb{R}^{+}$ by taking $r_j=h_j=f(j)$. Rearranging the terms in this identity, we obtain 
\begin{align}
\sum\limits_{n\leq x-1}\sum\limits_{j\leq x-n}f(n)f(n+j)&=\sum \limits_{2\leq n\leq x}f(n)\sum\limits_{m\leq n-1}f(m).\nonumber
\end{align}
Let $n+j=x-n$, then we deduce $x-2n=j$. It follows that $j\leq x-2$ if and only if  $1\leq n<\frac{x}{2}$. We can rewrite for the sum on the left-hand side
\begin{align}
\sum\limits_{n\leq x-1}\sum \limits_{j\leq x-n}f(n)f(n+j)&=\sum \limits_{n\leq x-1}\sum\limits_{x-2n=j}f(n)f(n+j)+\sum \limits_{n\leq x-1}\sum\limits_{\substack{j\leq x-n\\x-2n\neq j}}f(n)f(n+j)\nonumber \\&=\sum \limits_{n<\frac{x}{2}}f(n)f(x-n)+\frac{\mathcal{C}(x)}{x}\sum \limits_{n\leq x-1}\sum \limits_{j\leq x-n}f(n)f(n+j)\nonumber
\end{align}
where $0<\frac{\mathcal{C}(x)}{x}<1$. We deduce 
\begin{align}
\frac{\mathcal{D}(x)}{x}\sum \limits_{n\leq x-1}\sum \limits_{j\leq x-n}f(n)f(n+j)&=\sum \limits_{n<\frac{x}{2}}f(n)f(x-n)\nonumber
\end{align}
where $0<\frac{\mathcal{D}(x)}{x}=1-\frac{\mathcal{C}(x)}{x}<1$. Using Theorem \ref{identity 1}, we obtain
\begin{align}
\sum\limits_{n<\frac{x}{2}}f(n)f(x-n)&=\frac{\mathcal{D}(x)}{x}\sum \limits_{2\leq n\leq x}f(n)\sum \limits_{m\leq n-1}f(m)\nonumber
\end{align}
and the result follows immediately.
\end{proof}

\section{Application to the twin prime conjecture}

\begin{theorem}\label{proof}
There exists an absolute constant $\mathcal{C}>0$ such that 
\begin{align}
\sum\limits_{n\leq x}\Lambda(n)\Lambda(n+2)\geq (1+o(1))\frac{x}{2\mathcal{C}}.\nonumber
\end{align}
\end{theorem}

\begin{proof}
Using Theorem \ref{key 2}, we get 
\begin{align}
\sum\limits_{n\leq x}\Lambda(n)\Lambda(n+2)\geq\frac{1}{\mathcal{C}(2)x}\sum \limits_{2\leq n\leq x}\Lambda(n)\sum \limits_{m\leq n-1}\Lambda(m).\nonumber
\end{align}
Using the prime number theorem \cite{May} of the form 
\begin{align}
\sum \limits_{n\leq x}\Lambda(n)=(1+o(1))x\nonumber
\end{align}
establishes the lower bound.
\end{proof}
\bigskip

The lower bound in Theorem \ref{proof} solves the twin prime conjecture. The proposed method not only solves the twin prime conjecture, but is good in terms of its generality. It can be used to obtain lower bounds for a general class of correlated sums of the form 
\begin{align}
\sum \limits_{n\leq x}f(n)f(n+k)\nonumber
\end{align}
for a uniform $1\leq k\leq x$.

\section{Application to other correlated sums of type 1}

In this section, we apply Theorem \ref{key 2} to provide lower estimates of other correlated sums, but with the price of an implicit constant depending on the range of shift.

\begin{corollary}\label{divisor}
For a fixed $l_0>0$, there exist some $\mathcal{C}:=\mathcal{C}(l_0)>0$ such that 
\begin{align}
\sum \limits_{n\leq x}d(n)d(n+l_0)\geq (1+o(1))\frac{x\log^2 x}{2\mathcal{C}(l_0)}.\nonumber
\end{align}
\end{corollary}

\begin{proof}
Using Theorem \ref{key 2} and the crude estimate 
\begin{align}
\sum \limits_{n\leq x}d(n)=(1+o(1))x\log x\nonumber
\end{align}
in \cite{nathanson2000elementary} gives the lower bound.
\end{proof}

\begin{corollary}
For a fixed $k_0>0$ and for $l\geq 2$, there exists a constant $\mathcal{C}:=\mathcal{C}(k)>0$ such that 
\begin{align}
\sum \limits_{n\leq x}d_{l}(n)d_l(n+k)\geq (1+o(1))\bigg(\frac{1}{(l-1)!}\bigg)\bigg(1-\frac{1}{2(l-1)!}\bigg)\frac{x\log^{2(l-1)}x}{\mathcal{C}(k)}.\nonumber
\end{align}
\end{corollary}

\begin{proof}
We recall the weaker estimate for the $l$ th divisor function 
\begin{align}
\sum \limits_{n\leq x}d_l(n)=(1+o(1))\frac{1}{(l-1)!}x\log ^{l-1}x\nonumber
\end{align}
in \cite{nathanson2000elementary} where 
\begin{align}
d_l(n)=\sum \limits_{n_1\cdot n_2\cdots n_l=n}1.\nonumber
\end{align}
Using Theorem \ref{key 2} and using partial summation gives the lower bound. 
\end{proof}

\begin{corollary}
For a fixed $l_0>0$, there exists a constant $\mathcal{C}:=\mathcal{C}(l_0)>0$ such that 
\begin{align}
\sum\limits_{n\leq x}\phi(n)\phi(n+l_0)\geq (1+o(1))\frac{9}{2\pi^4}\frac{x^3}{\mathcal{C}(l_0)}.\nonumber
\end{align}
\end{corollary}

\begin{proof}
Using Theorem \ref{key 2} and the estimate 
\begin{align}
\sum\limits_{n\leq x}\phi(n)=(1+o(1))\frac{3}{\pi^2}x^2\nonumber
\end{align}
in \cite{tenenbaum2015introduction} gives the lower bound.
\end{proof}

\begin{corollary}\label{squarefree}
For a fixed $l_0>0$, there exists a constant $\mathcal{C}:=\mathcal{C}(l_0)>0$ such that 
\begin{align}
\sum\limits_{n\leq x}\mu^2(n)\mu^2(n+l_0)\geq (1+o(1))\frac{18}{\pi^4}\frac{x}{\mathcal{C}(l_0)}.\nonumber
\end{align}
\end{corollary}

\begin{proof}
Using Theorem \ref{key 2} and the estimate 
\begin{align}
\sum \limits_{n\leq x}\mu^2(n)=(1+o(1))\frac{6}{\pi^2}x\nonumber
\end{align}
in \cite{tenenbaum2015introduction} gives the lower bound.
\end{proof}
\bigskip

\section{Application to lower bound for two-point correlation of the  master function of type 1 and type 2}

In this section, we apply the area method developed to establish a lower bound for the two-point \textbf{type} $1$ correlation and an estimate for the \textbf{type} $2$  correlation of the master function. Here, $p,q$ denotes the prime numbers and $\Omega(\cdot)$ counts the number of prime factors with multiplicity.

\begin{definition}
    $$
    \Upsilon(n):=
    \begin{cases}
       \log p \quad \text{if} \quad n=p^2\\
       \log (pq) \quad \text{if} \quad n=pq, \quad p\neq q\\
       0 \quad \Omega(n)>2.
    \end{cases}
    $$
    We call $\Upsilon$ the \emph{master} function.
\end{definition}

\begin{lemma}\label{master}
Let $\Upsilon$ denote the master function. We have
\begin{align}
\sum\limits_{n\leq x}\Upsilon(n)=x\log\log x+O(x).\nonumber
\end{align}
\end{lemma}

\begin{proof}
For a proof, see, e.g., \cite{agama2017master}.
\end{proof}

\begin{theorem}\label{lower bound}
We have
\begin{align}
\sum\limits_{n\leq x}\Upsilon(n)\Upsilon(n+l_0)\geq (1+o(1))\frac{x}{2\mathcal{C}(l_0)}\log\log ^2x\nonumber
\end{align}
provided that $\Upsilon(n)\Upsilon(n+l_0)>0$.
\end{theorem}

\begin{proof}
Applying Theorem \ref{key 2} and using Lemma \ref{master}, we get
\begin{align}
\sum \limits_{n\leq x}\Upsilon(n)\Upsilon(n+l_0)&\geq \frac{1}{x\mathcal{C}(l_0)}\sum \limits_{2\leq n\leq x}\Upsilon(n)\sum \limits_{m\leq n-1}\Upsilon(m)\nonumber\\&=\frac{1}{x\mathcal{C}(l_0)}(1+o(1))\sum \limits_{2\leq n\leq x}\Upsilon(n)n\log\log n.\nonumber
\end{align}
By partial summation, we deduce
\begin{align}
\sum\limits_{2\leq n\leq x}\Upsilon(n)n\log\log n&=x\log\log x\sum \limits_{n\leq x}\Upsilon(n)-\int \limits_{2}^{x}(1+o(1))t(\log \log t)\bigg(\log\log t+\frac{1}{\log t}\bigg)dt\nonumber \\&=(1+o(1))x^2\log\log^2 x-(1+o(1))\int \limits_{2}^{x}t(\log\log t)\bigg(\log\log t+\frac{1}{\log t}\bigg)dt \nonumber \\&=(1+o(1))\frac{x^2}{2}\log\log^2 x.\nonumber
\end{align}
The claimed lower bound follows immediately.
\end{proof}

\begin{theorem}
Under the assumption 
\begin{align}
\frac{\sum\limits_{n\leq x}\sum \limits_{\substack{j\leq x-n\\j\neq x-2n}}\Upsilon(n)\Upsilon(n+j)}{\sum \limits_{n\leq x}\sum \limits_{j\leq x-n}\Upsilon(n)\Upsilon(n+j)}<1\nonumber
\end{align}
we have
\begin{align}
\sum\limits_{n\leq \frac{x}{2}}\Upsilon(n)\Upsilon(x-n)=(1+o(1))\frac{x}{2}\mathcal{D}(x)\log\log^2 x\nonumber
\end{align}
where $\mathcal{D}:=\mathcal{D}(x)>0$.
\end{theorem}

\begin{proof}
The result follows by applying the area method.
\end{proof}
\bigskip

\section{Application to estimates for the number of representations of an even number as a sum of two primes}

In this section, we apply the area method developed \ref{Area method 2} to obtain a weaker estimate for the number of representations of an even number as a sum of two primes under the assumption that the binary Goldbach conjecture is true.

\begin{theorem}\label{Goldbach}
Assume that the Goldbach conjecture is true. For any even $x\geq 6$, we have
\begin{align}
\sum\limits_{n\leq \frac{x}{2}}\Lambda(n)\Lambda(x-n)=(1+o(1))\frac{x}{2}\mathcal{D}(x)\nonumber
\end{align}
where $\mathcal{D}:=\mathcal{D}(x)>0$.
\end{theorem}

\begin{proof}
Under the assumption that the Goldbach conjecture is true, we deduce
\begin{align}
\frac{\sum \limits_{n\leq x}\sum \limits_{\substack{j\leq x-n\\j\neq x-2n}}\Lambda(n)\Lambda(n+j)}{\sum \limits_{n\leq x}\sum \limits_{j\leq x-n}\Lambda(n)\Lambda(n+j)}<1.\nonumber
\end{align}
Applying the area method, there exists some $\mathcal{D}:=\mathcal{D}(x)>0$ with $\mathcal{D}(x)<x$ such that \begin{align}
\sum\limits_{n\leq \frac{x}{2}}\Lambda(n)\Lambda(x-n)=\frac{\mathcal{D}(x)}{x}\sum \limits_{2\leq n\leq x}\Lambda(n)\sum\limits_{m\leq n-1}\Lambda(m).\nonumber
\end{align}
Using the prime number theorem \cite{May} in the form 
\begin{align}
\sum\limits_{n\leq x}\Lambda(n)=(1+o(1))x\nonumber
\end{align}
we obtain 
\begin{align}
\sum\limits_{2\leq n\leq x}\Lambda(n)\sum\limits_{m\leq n-1}\Lambda(m)&=(1+o(1))\frac{x^2}{2}.\nonumber
\end{align}
The result follows immediately.
\end{proof}

\section{Application to other correlated sums of type 2}

In this section, we apply the area method \ref{Area method 2} to obtain estimates for various correlated sums of \textbf{type} 2.

\begin{theorem}
We have
\begin{align}
\sum \limits_{n\leq\frac{x}{2}}d(n)d(x-n)=\mathcal{D}(1+o(1))\frac{x\log^2x}{2}\nonumber
\end{align}
where $\mathcal{D}:=\mathcal{D}(x)>0$.
\end{theorem}

\begin{proof}
Since the divisor function is non-vanishing on the integers, we deduce
\begin{align}
\frac{\sum \limits_{n\leq x}\sum \limits_{\substack{j\leq x-n\\x-2n\neq j}}d(n)d(n+j)}{\sum \limits_{n\leq x}\sum \limits_{j\leq x-n}d(n)d(n+j)}<1\nonumber.
\end{align}
Using the area method (Theorem \ref{Area method 2}), there exist some $\mathcal{D}:=\mathcal{D}(x)$ with $0<\mathcal{D}(x)<x$ such that \begin{align}
\sum\limits_{n\leq \frac{x}{2}}d(n)d(x-n)=\frac{\mathcal{D}(x)}{x}\sum \limits_{2\leq n \leq x}d(n)\sum\limits_{m\leq n-1}d(m).\nonumber
\end{align}
Using the weaker estimate 
\begin{align}
\sum\limits_{n\leq x}d(n)=(1+o(1))x\log x\nonumber
\end{align}
in \cite{nathanson2000elementary}, we obtain 
\begin{align}
\sum \limits_{2\leq n\leq x}d(n)\sum \limits_{m\leq n-1}d(m)&=(1+o(1))\frac{x^2\log^2 x}{2}.\nonumber
\end{align}
The estimate is obtained using the area method.
\end{proof}

\begin{theorem}
We have 
\begin{align}
\sum\limits_{n\leq \frac{x}{2}}\phi(n)\phi(x-n)=(1+o(1))\mathcal{D}\frac{9}{2\pi^4}x^3\nonumber
\end{align}
where $\mathcal{D}:=\mathcal{D}(x)>0$.
\end{theorem}

\begin{proof}
Since $\phi(n)$ is non-vanishing on the integers, we deduce
\begin{align}
\frac{\sum \limits_{n\leq x}\sum \limits_{\substack{j\leq x-n\\x-2n\neq j}}\phi(n)\phi(n+j)}{\sum \limits_{n\leq x}\sum \limits_{j\leq x-n}\phi(n)\phi(n+j)}<1.\nonumber
\end{align}
By the area method (Theorem \ref{Area method 2}), there exist some $\mathcal{D}:=\mathcal{D}(x)>0$ with $\mathcal{D}(x)<x$ such that \begin{align}
\sum \limits_{n\leq\frac{x}{2}}\phi(n)\phi(x-n)=\frac{\mathcal{D}(x)}{x}\sum \limits_{2\leq n \leq x}\phi(n)\sum \limits_{m\leq n-1}\phi(m).\nonumber
\end{align}
Using the estimate 
\begin{align}
\sum \limits_{n\leq x}\phi(n)=(1+o(1))\frac{3}{\pi^2}x^2\nonumber
\end{align}
in \cite{nathanson2000elementary}, we obtain
\begin{align}
\sum\limits_{2\leq n\leq x}\phi(n)\sum\limits_{m\leq n-1}\phi(m)=(1+o(1))\frac{9}{2\pi^4}x^4.\nonumber
\end{align}
The claimed estimate is deduced using the area method.
\end{proof}

\begin{theorem}
We have
\begin{align}
\sum\limits_{n\leq \frac{x}{2}}d_l(n)d_l(x-n)=(1+o(1))\mathcal{D}\bigg(\frac{1}{(l-1)!}\bigg)\bigg(1-\frac{1}{2(l-1)!}\bigg)x\log^{2(l-1)}x\nonumber
\end{align}
where $\mathcal{D}:=\mathcal{D}(x)>0$ and where 
\begin{align}
d_l(n):=\sum \limits_{n_1n_2\cdots n_l=n}1.\nonumber
\end{align}
\end{theorem}

\begin{proof}
We observe
\begin{align}
\frac{\sum \limits_{n\leq x}\sum \limits_{\substack{j\leq x-n\\j\neq x-2n}}d_l(n)d_l(n+j)}{\sum \limits_{n\leq x}\sum \limits_{j\leq x-n}d_l(n)d_l(n+j)}<1.\nonumber
\end{align}
Using the area method (Theorem \ref{Area method 2}), there exists some $\mathcal{D}:=\mathcal{D}(x)>0$ with $\mathcal{D}(x)<x$ such that \begin{align}
\sum\limits_{n\leq \frac{x}{2}}d_l(n)d_l(x-n)=\frac{\mathcal{D}(x)}{x}\sum\limits_{2\leq n\leq x}d_l(n)\sum \limits_{m\leq n-1}d_l(m).\nonumber
\end{align}
Using the estimate 
\begin{align}
\sum\limits_{n\leq x}d_l(n)=(1+o(1))\frac{1}{(l-1)!}x\log ^{l-1}x,\nonumber
\end{align}
\cite{nathanson2000elementary}, we deduce
\begin{align}
\sum\limits_{2\leq n \leq x}d_l(n)\sum \limits_{m\leq n-1}d_l(m)=(1+o(1))\bigg(\frac{1}{(l-1)!}\bigg)\bigg(1-\frac{1}{2(l-1)!}\bigg)x^2\log^{2(l-1)}x.\nonumber
\end{align}
The claimed estimate is deduced using the area method.
\end{proof}
\bigskip

\section{Application to the global distribution of integers with $\Omega(n)=2$}

The lower bounds of the correlations of the arithmetic functions determine their local and global distribution. Theorem \ref{lower bound} gives 
\begin{align}
\sum\limits_{n\leq x}\Upsilon(n)\Upsilon(n+l_0)\geq (1+o(1))\frac{x}{2\mathcal{C}(l_0)}\log \log ^2x,\nonumber
\end{align} 
provided that $\sum\limits_{n\leq x}\Upsilon(n)\Upsilon(n+l_0)>0$. Thus, for some shift in the range $[1,x]$ the correlation can be made arbitrarily large by taking the right hand side arbitrarily large. This implies that there are infinitely many pairs of the form $(n,n+l_0)$ such that $n$ and $n+l_0$ each have exactly two prime factors.

\footnote{
\par}

\bibliographystyle{amsplain}

\end{document}